\documentclass[a4paper,11pt]{article}
\usepackage[utf8]{inputenc}
\usepackage[T1]{fontenc}

\usepackage{amsmath,amsthm,amssymb,mathtools}
\usepackage{longtable}
\usepackage{graphicx}
\usepackage{multicol}
\usepackage{hyperref}
\def\N{\mathbb N}
\def\Z{\mathbb Z}

\def\R{\mathbb R}

\def\A{\mathbb A}
\def\B{\mathbb B}

\def\T{\mathcal T}
\def\i{\overline{1}}
\def\b{\boldsymbol}
\def\p{1m\overline mn}
\def\m{\i\overline mm\overline n}

\def\int{\mathrm{int}}
\def\fr{\mathrm{fr}}

\newcommand{\ZMB}{\mathbb Z_{-\beta}}

\newcommand{\TMB}{T_{-\beta}}
\newcommand{\DMBL}{d_{-\beta}(\ell)}
\newcommand{\DMBRS}{d_{-\beta}^*(\ell+1)}

\newcommand{\rozvojm}[1]{\langle #1\rangle_{-\beta}}

\newcommand{\aprec}{\prec_{\mathrm{alt}}}
\newcommand{\apreceq}{\preceq_{\mathrm{alt}}}

\def\Finmb{\operatorname{Fin}(-\beta)}
\newtheorem{lem}{Lemma}
\newtheorem{thm}[lem]{Theorem}

\newtheorem{de}[lem]{Definition}
\newtheorem{ex}[lem]{Example}

\author{Tom\'a\v s V\'avra\\ Department of Mathematics FNSPE\\ Czech Technical University in Prague} %\thanks{Email: \email{t.vavra@seznam.cz}.}}

\title{On the Finiteness property of negative cubic Pisot bases}
\begin{document}
\maketitle
\begin{abstract}
We study arithmetical aspects of Ito-Sadahiro number systems with negative base. We show that the bases $-\beta<-1$, where $\beta$ is zero of $x^3-mx^2-mx-m,\ m\in\mathbb N,$ possess the so-called finiteness property. For the Tribonacci base $-\gamma,$ zero of $x^3-x^2-x-1$, we present an effective algorithm for addition and subtraction. In particular, we present a finite state transducer performing these operations. As a consequence of the structure of the transducer, we determine the maximal number of fractional digits arising from addition or subtraction of two \hbox{$(-\gamma)$-integers.}
\end{abstract}
%%%%%%%%%%%%%%%%%%%%%%%%%%%%%%%
\section{Introduction}
In the study of numeration systems with non-integer base, after the Golden ratio,
probably the most attention was given to its first generalization, the so-called Tribonacci constant $\gamma$, the real root of $x^3-x^2-x-1$. Multiplication in systems corresponding to \hbox{$\gamma$-expansions} was studied Grabner et al. \cite{GR} and later Messaoudi \cite{MES}, providing bounds on the number of fractional digits of the product of two \hbox{$\gamma$-integers.} Similar results for addition were obtained in \cite{AmFrMaPe, BER}.

An important characteristics of a numeration system is the so-called finiteness property (F), namely that the set of finite expansions forms a ring. Property (F) for $\gamma$-expansions
follows from a general result by Frougny and Solomyak \cite{FRSO}. This has an interesting consequence for proving connectedness of certain discrete planes \cite{berthe}.

Numeration systems with negative non-integer base $-\beta<-1$ received a non-negligible
attention since the paper [8] of Ito and Sadahiro in 2009. Certain arithmetic aspects seem
to be analogous to those for positive base systems \cite{FRLAI11, MPV11}. However, the construction of concrete algorithms is much more complicated in the former case.
This might be due to the fact that while $\beta$-expansions are lexicographicaly greatest among all the representations of a given number, the $(-\beta)$-expansions are not extremal with respect to neither lexicographical nor alternate lexicographical ordering, which is more natural for negative base, see~\cite{DK}.

For example, in positive base systems, Property (F) has been proven for large classes of Pisot numbers \cite{AKI, FRSO, HOL}. On the other hand, so far the only numbers known to possess analogous Property (-F) are negative bases $-\beta$ where $\beta$ is a root of $x^2-mx+n$, $m-2\geq n\geq 1$, $m,n\in{\mathbb N}$, as shown in \cite{MPV11}.

The main results of this paper are proving Property (-F) for zeros $>1$ of $x^3-mx^2-mx-m,\ m\in\N$, and an explicit description of a finite state transducer performing addition and subtraction in negative Tribonacci base $-\gamma$. The existence of such a transducer was proved in [5] for any negative Pisot base. However, the proof is not constructive, and does not provide any simple procedure to construct such a transducer for a specific base. This result has implications for the possibility
of modelling certain related sets in the frame of the so called cut-and project scheme, see~\cite{HePe}.
We also give a precise bound on the number of fractional digits arising from addition and subtraction of $(-\gamma)$-integers.
%%%%%%%%%%%%%%%%%%%%%%%%%%%%%%%%% 
\section{Preliminaries}
The \emph{$(-\beta)$-expansions} considered by Ito-Sadahiro in \cite{IS09} are 
an analogue of \hbox{\emph{$\beta$-expansions,}} introduced by R\'enyi in \cite{REN}, which use a
negative base. Instead of
defining the expansions of numbers from $[0,1)$, the unit
interval $[\ell,\ell+1)$ with $\ell=\frac{-\beta}{\beta+1}$ 
was chosen. For $-\beta<-1$, any $x\in[\ell,\ell+1)$ has a
unique expansion of the form $d_{-\beta}(x)=x_1x_2x_3\cdots$
defined by
\begin{displaymath}
x_i=\lfloor-\beta\TMB^{i-1}(x)-\ell\rfloor,\quad\text{
where }\quad\TMB(x)=-\beta x-\lfloor-\beta x-\ell\rfloor\,.
\end{displaymath}
For any
 $x\in[\ell,\ell+1)$ we obtain an
infinite word from $A^\N=\{0,1,\ldots,\lfloor\beta\rfloor\}^\N$.

Another analogous concept is the \emph{$(-\beta)$-admissibility}, which
characterizes all digit strings over $A$ that are 
$(-\beta)$-expansion of some number. The lexicographic condition,
similar to the one by Parry, was also proved in~\cite{IS09}. Ito
and Sadahiro proved that a digit string $x_1x_2x_3\cdots\in A^\N$
is \emph{$(-\beta)$-admissible} (or, if no confusion is possible, just
admissible) if and only if it fulfills the lexicographic condition
\begin{equation*}
\label{eq_admissibilita_is} \DMBL\apreceq
x_ix_{i+1}x_{i+2}\cdots\aprec\DMBRS=\lim_{y\to
\ell+1_-}d_{-\beta}(y)\ \ \text{  for all }i\geq 1\,.
\end{equation*}
Here, the limit is taken over the product topology on $A^\N$ and
 $\aprec$ stands for \emph{alternate lexicographic ordering} defined
as follows: 
\begin{displaymath}
u_1u_2\cdots\aprec v_1v_2\cdots\ \Leftrightarrow\
(-1)^k(u_k-v_k)<0,\quad k=\min\{n\in\N :\ u_n\neq v_n\}.
\end{displaymath}
In analogy with $\beta$-numeration, the alternate
ordering corresponds to the ordering
on reals in $[\ell,\ell+1)$, i.e. $x<y\ \Leftrightarrow\
d_{-\beta}(x)\aprec d_{-\beta}(y)$.

The reference digit strings $\DMBL$ and $\DMBRS$ play the same
role for $(-\beta)$-expansions as R\'enyi expansions of unity for
$\beta$-expansions. While $\DMBL$ is obtainable directly from the
definition, the following rule (proved in~\cite{IS09}) can be used
for determining $\DMBRS$:
\begin{displaymath}
\DMBRS=\begin{cases}
(0l_1\cdots l_{q-1}(l_q-1))^\omega & \text{if $\DMBL=(l_1l_2\cdots l_q)^\omega$ for $q$ odd,} \\
0\DMBL & \text{otherwise.}
\end{cases}
\end{displaymath}
\begin{ex}
When $\beta>1$ is zero of $x^3-mx^2-mx-m,\ m\in\N$, we have $\DMBL=m0m^\omega.$
The lexicographic condition is 
\begin{displaymath}
m0m^\omega\apreceq x_ix_{i+1}\cdots\aprec 0m0m^\omega.
\end{displaymath}
Therefore the sequences from $A^\N$ with substrings $m0m^{2k-1}n,\ n<m,\ k\in\N$, or $0m0m^\omega$, are not admissible. Conversely, any sequence not containing those substrings is admissible.
\end{ex}
We can now recall the definition of
$(-\beta)$-expansions for all reals.

\begin{de}\label{de_minusbetarozvoje}
Let $-\beta<-1$, $x\in\R$. Let $k\in\N$ be minimal such that
$\frac{x}{(-\beta)^k}\in(\ell,\ell+1)$ and
$d_{-\beta}\Big(\frac{x}{(-\beta)^k}\Big)=x_1x_2x_3\cdots$. Then
the $(-\beta)$-expansion of $x$ is defined as
\begin{displaymath}
\rozvojm{x}=\begin{cases}
x_1\cdots x_{k-1}x_k\bullet x_{k+1}x_{k+2}\cdots & \text{ if $k\geq 1$,}\\
0\bullet x_1x_2x_3\cdots & \text{ if $k=0$.}
\end{cases}
\end{displaymath}
\end{de}

Similarly as in a positive base numeration, the set of
\emph{$(-\beta)$-integers} can be defined using the notion of
$\rozvojm{x}$. Since the base is negative, we can represent
any real number without the need of a minus sign. 

\begin{de}\label{de_minusbetacela}
Let $\beta>1$. Then the sets of $(-\beta)$-integers and of numbers with finite \hbox{$(-\beta)$-expansions} are defined as
\begin{displaymath}
\ZMB=\{x\in\R : \rozvojm{x} = x_k\cdots x_1x_0\bullet
0^\omega\}\ =\ \bigcup_{i\geq 0}(-\beta)^i\TMB^{-i}(0)\,,
\end{displaymath}
\begin{displaymath}
\Finmb=\{x\in\R : \rozvojm{x} = x_k\cdots x_1x_0\bullet
x_{-1}\cdots x_{-n}0^\omega\}\ =\ \bigcup_{i\geq 0}(-\beta)^{ -i}\ZMB\,.
\end{displaymath}
\end{de}
We say that $\beta$ has \emph{Property $\operatorname{(-F)}$} if $\Finmb$ is a ring, in other words, if 
$\mathrm{Fin}(-\beta)=\Z[\beta].$

We will be interested the length of the \emph{fractional part} of an expansion. It means that if $\langle x\rangle_{-\beta} = x_k\cdots x_0\bullet x_{-1}\cdots x_{-n}0^\omega$ with $x_{-n}\neq 0,$
we denote the length of the fractional part of $x$ as $\fr(x)=n.$ We set $\fr(x)=+\infty$ if $x\notin \Finmb$. The suffix $0^\omega$ might be omitted.

It was shown in \cite{DMV} that the number system with negative base $-\beta$ possesses interesting
properties when $\beta>1$ is set to be the zero of
\begin{equation}\label{baze}
x^k-mx^{k-1}-\dots-mx-n,\ m\geq n\geq1\quad\text{and }\ m=n\text{ for } k \text{ even.}
\end{equation}
In particular, 
\begin{equation}\label{confl}
\text{if}\quad x=\sum_{i=-m}^k x_i(-\beta)^i,\ x_i\in A\quad \text{then}\quad\fr(x)\leq m.
\end{equation}
Let us recall that an analogous result for $\beta$-expansions comes from Ch.~Frougny~\cite{FR}.

We will also often work with \emph{$(-\beta)$-representations}. Note that a representation differs from the expansion by the fact that it does not have to be admissible or even over the same alphabet. In other words, the representation $x_kx_{k-1}\cdots x_0\bullet x_{-1}\cdots$ corresponds to the number $\sum_{i\leq k} x_i(-\beta)^i$ for $x_i\in \Z.$

%Often observed property is a number of fractional digits arising from arithmetical operations. For that purpose, we define the quantity
%$$L_\oplus(-\beta)=\sup\{\fr(x\pm y)\in\N : x,y\in\Z_{-\beta}\}.$$

%%%%%%%%%%%%%%%%%%%%%%%%%%%
\section{The arithmetics of $(-\gamma)$-expansions}
\begin{figure}\centering
\includegraphics[width=10cm]{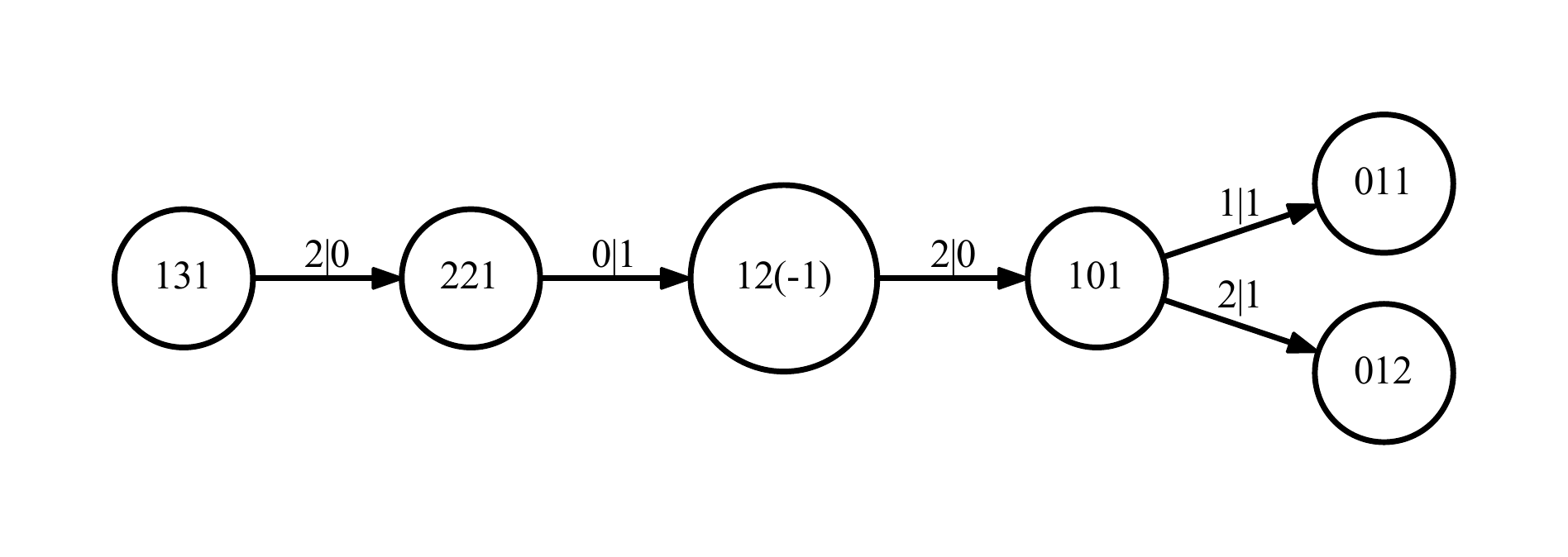}
\caption{Transitions leading to the state $101$ which cannot be escaped by reading zero}
\label{subgraph}
\end{figure}
In Section \ref{transducers}, transducers $\T_+$ and $\T_-$ are defined. The former acts on the set $\{0,1,2\}^\N$ and if the input is a digit-wise sum of two $(-\gamma)$-integers $x,y\in\Z_{-\gamma},$ then the output sequence is a representation of $x+y$ over $\{0,1\}.$ The output sequence is not necessarily an expansion. The Ito-Sadahiro $(-\beta)$-expansions are not a redundant numeration system, thus addition is not computable by a finite state transducer while respecting the admissibility condition. Also, $\T_+$ is not defined on the whole set $\{0,1,2\}^\N,$ namely, the substring $2020$ may not be recognized (see Figure~\ref{subgraph}). Nevertheless, reading $2020$ corresponds to the occurence of forbidden factors $1010$ in both $x$ and $y$. Thus, it is sufficient to assume that the string representing either $x$ or $y$ is admissible.

The transducer $\T_-$ transforms a $(-\gamma)$-representation over the alphabet $\{0,\i\}$ to the alphabet $\{0,1\}$. The symbol $\overline a$ stands for $-a.$ Then, given a representation of $x$ over $\{0,1\}$, the representation of $-x$ over $\{0,1\}$ can be computed as $\T_-(\overline x)$ where $\overline x$ stands for the string $x$ whose digits \lq\lq{}are reversed\rq\rq{}, i.e. an image under the morphism $0\mapsto 0, 1\mapsto \i.$

Even though the transducers act on sequences, the output sequence can be regarded as a representation. For, from the contruction it will be clear where the fractional point lies. We denote by  
$\fr_{\T_+}(x)$ and  $\fr_{\T_-}(x)$ the lengths of fractional parts of the output \hbox{$(-\gamma)$-representations,} i.e. if $\T_+(x)=\cdots\bullet z_{-1}\cdots z_{-n}0^\omega,$ $z_{-n}\neq 0,$ then $\fr_{\T_+}(x)=n.$
Note that $\fr(x)\leq\fr_{\T_+}(x),$ as follows from~\eqref{confl}.

Before we prove Theorem~\ref{teorem}, we present a technical lemma whose proof can be found in Section~\ref{transducers}.	

\begin{lem}\label{lemmatecko}
Let $z=z_k\cdots z_0\bullet0^\omega$, $z_i\in\{0,1,2\},$ be a digit-wise sum of two
\hbox{$(-\gamma)$-integers.}
  \begin{enumerate}
    \item We have $\fr(z)\leq 6.$
    \item If $z_0\neq2$, then $\fr_{\T_+}(z)\leq 5. $
    \item If $\fr_{\T_+}(z)\geq2,$ then $\T_+(z)$ has suffix $0010^\omega$.
    \item If $y\in\Z_{-\gamma}$, then $\fr_{\T_-}(\overline y)\leq 3.$
    \item Let $y\in\Z_{-\gamma}$ and $\fr_{\T_-}(\overline y)\geq2.$ Then $\T_-(\overline y)$ has suffix $0010^\omega$.
  \end{enumerate}
\end{lem}
 
\begin{thm}\label{teorem}
  Let $x,y\in\Z_{-\gamma}.$ Then $\fr(x+y),\fr(-x)\leq 6.$
\end{thm}
\begin{proof}
%The Property $(-\mathrm{F})$ is a direct consequence of Lemma~\ref{lemmatecko} and the identity $\i=110\bullet1.$
The fact that $\fr(x+y)\leq 6$ is a direct consequence of item 1 of Lemma~\ref{lemmatecko}.
Let us focus on $\fr(x-y).$ We distinguish the cases corresponding to the values $\fr_{\T_-}(\overline y)\in\{0,1,2,3\}.$

Obviously, we have $x-y=x+\T_-(\overline y),$ and if 
$\fr_{\T_-}(\overline y)=0$, then by Lemma~\ref{lemmatecko}, item 1, it holds $\fr(x-y)\leq6.$
  \begin{enumerate}
    \item If $\fr_{\T_-}(\overline y)=1,$ then
$x-y=x+\T_-(\overline y)=z_k\cdots z_0\bullet1,\ z_i\in\{0,1,2\}$. 
By Lemma~\ref{lemmatecko}, item 2, we have
$\fr_{\T_+}\left((-\beta)(x-y)\right)\leq 5$, thus $\fr(x-y)\leq6.$

    \item Let $\fr_{\T_-}(\overline y)=2.$ Then by item 3 of the lemma, $\T_-(\overline y)$ has suffix $001$ and we can write 
$$x-y=z_k\cdots z_0\bullet+0\bullet 01=z+0\bullet01,\ z_0\neq2.$$ 
We now apply $\T_+$ on $z$.
It holds $\fr_{\T_+}(z)\leq5$ and in case
$\fr_{\T_+}(z)\in\{0,1\}$ we directly obtain $\fr(x-y)\leq 2.$

If $\fr_{\T_+}(z)=2$, then $z_0=0$ and thus
$$x-y=z+0\bullet01=\T_+(z)+0\bullet 01=\cdots 0\bullet02=\cdots 1\bullet111,$$
i.e. $\fr(x-y)\leq 3.$ Here we used the representation of zero $0=11\i1=\i\i1\i$, that follows from the minimal polynomial for $\gamma$.

In case $\fr_{\T_+}(z)\in\{3,4\},$ we have by item 2 of the lemma that $\T_+(z)=\cdots\bullet 001$ or $\T_+(z)=\cdots\bullet z_{-1}001$, and therefore $\fr(x-y)\leq 4.$

If $\fr_{\T_+}(z)=5$
then either $x-y=\cdots\bullet z_{-1}1001,\ z_{-1}\in\{0,1\},$ or one of the following happens.
\begin{enumerate}
\item $x-y=\cdots\bullet 12001=\cdots\bullet0^\omega,$ or
\item $$x-y=\cdots0(01)^k\bullet 02001=
\begin{cases} 
  \cdots1\bullet1001&\text{for }k=0, \\
  \cdots110\bullet0&\text{for }k=1.\\
\end{cases}$$
Using the rewriting rules 
$$0=11(\i0\i)(101\i0\i)^m1\i=11(\i0\i)(101\i0\i)^m(101)\i1,\quad  m\geq0,$$
one can verify that $\fr(x-y)\leq 5$ also holds for all $k\geq2$.
\end{enumerate}
    \item If $\fr_{\T_-}(\overline y)=3,$ then the statement clearly holds for 
$\fr_{\T_+}(z)\in\{0,1,2,4,5\}.$
The cases $\fr_{\T_+}(z)\in\{3,6\}$ can be solved analogously to the previous case.
  \end{enumerate}
\end{proof}

The maximum value of fractional digits arising from arithmetical operations is sometimes denoted as $L_\oplus(-\beta)$ (addition and subtraction) and $L_\otimes(-\beta)$ (multi\-plication), cf.~\cite{MPV11}.
In fact, Theorem \ref{teorem} gives $L_\oplus(-\gamma)\leq 6.$ The inequality $L_\oplus(-\gamma)\geq 6$ is shown in  Example~\ref{priklad}. Altogether, we have $L_\oplus(-\gamma)=6.$
%%%%%%%%%%%%%%%%%%%%%%%%%%%%%%%%%%%%%
\section{Transducers $\T_+$ and $\T_-$}\label{transducers}
The aim of this section is to demonstrate Lemma~\ref{lemmatecko}, which was crucial in proving Theorem~\ref{teorem}. The main element for the demonstration are transducers $\T_+$ and $\T_-$ which are defined below.
\subsection{Transducer $\T_+$}
The idea of the construction of the adding automaton is the following: First, do a digit-wise sum of two numbers. Then, read this representation over the alphabet $\{0,1,2\}$ from left to right, and, using the rewriting rule $0=11\i1,$ rewrite it to the alphabet $\{0,1\}.$  We show that we can obtain a letter-to-letter transducer, i.e. a transducer that reads and outputs only one letter at a time. However, for simplicity, we present a transducer that is not letter-to-letter, but can be modified to such.

We define the tranducer $\T_+=(Q,\mathbb A,\mathbb B,E,I)$ where
\begin{itemize}
  \item $Q\subset \{\overline 2,\overline 1,0,1,2,3\}^3$ is the set of states;
  \item $I=000$ is the initial state;
  \item $\A=\{0,1,2\}$ is the input alphabet;
  \item $\B=\{0,1\}$ is the output alphabet;
  \item $E\subset\{(q,(u,v),q'):\ q,q'\in Q,\ (u,v)\in \A^*\times \B^*\}$ are the labeled edges defined in the list below.
\end{itemize}
We will denote the edge $(q,(u,v),q')$ as the transition $q|u\rightarrow v|q'.$ The transducer can be seen as a sliding window of length 4 (or 5 in some cases) that changes the string inside of itself and moves to the right while putting its left most symbol to the output.

One can verify that the transitions defined below do not change the numerical value of the input string since the right side is obtained by adding a representations of zero. That is, if $q|u\rightarrow v|q'$ then $vq'$ can be obtained from $qu$ by adding the strings $11\i1$ or $\i\i1\i$ (possibly two times). Moreover, with one exception, the transitions from any state are defined for any input letter from $A.$ The exception is the state $101$ that cannot be escaped by reading symbol $0.$ Nevertheless, the only path to the state $101$ leads from $131$ by reading the string $202$ on the input 
(see Figure~\ref{subgraph}). Reading $2020$ would mean that both $x$ and $y$ contain forbidden string $1010$. Therefore we assume that at least one summand does not contain $1010.$
%%%%%%%%%%%%%%%%%%%%%%%%%%%%%
\begin{ex}\label{priklad}
We will show how $\T_+$ acts on the number $212\bullet.$ First, we add the leading factor $000$ 
and the tailing factor $000000000$ to the representation. Below is the sliding window representation and on the right side are the corresponding transitions of $\T_+.$ Input and output symbols are in bold.

\setlength{\arraycolsep}{3.5pt}
\begin{displaymath}
\begin{array}{cccccc @{\ \ \bullet\ \ }cccccccccc|c}
[0&0&0&\b2]&1&2&0&0&0&0&0&0&0&0&0&0&000|2\rightarrow 0|002\\
\b0&[0&0&2&\b1]&2&0&0&0&0&0&0&0&0&0&0&002|1\rightarrow 1|112 \\
0&\b1&[1&1&2&\b2]&0&0&0&0&0&0&0&0&0&0& 112|2\rightarrow 1|122 \\
0&1&\b1&[1&2&2&\b0]&0&0&0&0&0&0&0&0&0&  122|0\rightarrow 1|220\\
0&1&1&\b1&[2&2&0&\b0]&0&0&0&0&0&0&0&0&  220|0\rightarrow 1|11\i\\
0&1&1&1&\b1&[1&1&\i&\b0]&0&0&0&0&0&0&0&  11\i|0\rightarrow0|00\i\\
0&1&1&1&1&\b0&[0&0&\i&\b0]&0&0&0&0&0&0& 00\i|0\rightarrow0|0\i0\\
0&1&1&1&1&0&\b0&[0&\i&0&\b0]&0&0&0&0&0& 0\i0|0\rightarrow1|0\i1\\
0&1&1&1&1&0&0&\b1&[0&\i&1&\b0]&0&0&0&0& 0\i1|0\rightarrow1|001\\
%0&1&1&1&1&0&0&1&1&[0&0&1&\b0]&0&0&0 \\
0&1&1&1&1&0&0&1&\b1&[0&0&1&\b0&\b0]&0&0& 001|00\rightarrow00|100\\
0&1&1&1&1&0&0&1&1&\b0&\b0&[1&0&0&\b0]&0& 100|0\rightarrow1|000\\
0&1&1&1&1&0&0&1&1&0&0&\b1&[0&0&0&\b0]& 000|0\rightarrow0|000\\
\end{array}
\end{displaymath}
Hence it holds that $212\bullet = 11110\bullet011001.$ The latter representation is admissible from which it follows that $\fr(2\gamma^2-\gamma+2)=6.$
\end{ex}
%%%%%%%%%%%%%%%%%%%%%%%%%%%%%%%%%
The transducer $\T_+$ is defined by the following list of transitions:

{
\catcode`\"=13
\catcode`\(=13
\def(-#1){\overline #1}
\def"#1"#2-> "#3"#4"#5|#6"];{\noindent$\phantom{333|33}\mathllap{#1 | #5} \rightarrow #6 | #3$\par}
\begin{centering}
\begin{multicols}{3}
"000":n -> "000":s [ penwidth=2, label ="0|0"];
"000" -> "001" [ penwidth=2, label ="1|0"];
"000" -> "002" [ penwidth=2, label ="2|0"];
"001" -> "100" [ penwidth=2, label ="00|00"];
"001" -> "011" [ penwidth=2, label ="01|11"];
"001" -> "012" [ penwidth=2, label ="02|11"];
"001" -> "011" [ penwidth=2, label ="1|0"];
"001" -> "012" [ penwidth=2, label ="2|0"];
"002" -> "111" [ penwidth=2, label ="0|1"];
"002" -> "112" [ penwidth=2, label ="1|1"];
"002" -> "131" [ penwidth=2, label ="21|11"];
"002" -> "132" [ penwidth=2, label ="22|11"];
"002" -> "220" [ penwidth=2, label ="20|00"];
"003" -> "121" [ penwidth=2, label ="0|1"];
"003" -> "122" [ penwidth=2, label ="1|1"];
"003" -> "123" [ penwidth=2, label ="2|1"];
"00(-1)" -> "0(-1)0" [ penwidth=2, label ="0|0"];
"00(-1)" -> "0(-1)1" [ penwidth=2, label ="1|0"];
"00(-1)" -> "0(-1)2" [ penwidth=2, label ="2|0"];
"011" -> "110" [ penwidth=2, label ="0|0"];
"011" -> "111" [ penwidth=2, label ="1|0"];
"011" -> "112" [ penwidth=2, label ="2|0"];
"012" -> "120" [ penwidth=2, label ="0|0"];
"012" -> "121" [ penwidth=2, label ="1|0"];
"012" -> "122" [ penwidth=2, label ="2|0"];
"013" -> "21(-1)" [ penwidth=2, label ="00|00"];
"013" -> "120" [ penwidth=2, label ="01|11"];
"013" -> "121" [ penwidth=2, label ="02|11"];
"013" -> "131" [ penwidth=2, label ="1|0"];
"013" -> "132" [ penwidth=2, label ="2|0"];
"023" -> "230" [ penwidth=2, label ="0|0"];
"023" -> "231" [ penwidth=2, label ="1|0"];
"023" -> "232" [ penwidth=2, label ="2|0"];
"0(-1)0" -> "0(-1)1" [ penwidth=2, label ="0|1"];
"0(-1)0" -> "0(-1)2" [ penwidth=2, label ="1|1"];
"0(-1)0" -> "111" [ penwidth=2, label ="20|00"];
"0(-1)0" -> "112" [ penwidth=2, label ="21|00"];
"0(-1)0" -> "023" [ penwidth=2, label ="22|11"];
"0(-1)1" -> "001" [ penwidth=2, label ="0|1"];
"0(-1)1" -> "002" [ penwidth=2, label ="1|1"];
"0(-1)1" -> "003" [ penwidth=2, label ="2|1"];
"0(-1)2" -> "011" [ penwidth=2, label ="0|1"];
"0(-1)2" -> "012" [ penwidth=2, label ="1|1"];
"0(-1)2" -> "013" [ penwidth=2, label ="2|1"];
"0(-1)(-1)" -> "0(-1)1" [ penwidth=2, label ="00|00"];
"0(-1)(-1)" -> "0(-1)2" [ penwidth=2, label ="01|00"];
"0(-1)(-1)" -> "(-1)03" [ penwidth=2, label ="02|11"];
"0(-1)(-1)" -> "(-1)(-1)1" [ penwidth=2, label ="1|0"];
"0(-1)(-1)" -> "(-1)(-1)2" [ penwidth=2, label ="2|0"];
"100" -> "000" [ penwidth=2, label ="0|1"];
"100" -> "001" [ penwidth=2, label ="1|1"];
"100" -> "002" [ penwidth=2, label ="2|1"];
"101" -> "011" [ penwidth=2, label ="1|1"];
"101" -> "012" [ penwidth=2, label ="2|1"];
"10(-1)" -> "0(-1)0" [ penwidth=2, label ="0|1"];
"10(-1)" -> "0(-1)1" [ penwidth=2, label ="1|1"];
"10(-1)" -> "0(-1)2" [ penwidth=2, label ="2|1"];
"110" -> "100" [ penwidth=2, label ="0|1"];
"110" -> "100" [ penwidth=2, label ="10|00"];
"110" -> "011" [ penwidth=2, label ="11|11"];
"110" -> "012" [ penwidth=2, label ="12|11"];
"110" -> "011" [ penwidth=2, label ="2|0"];
"111" -> "110" [ penwidth=2, label ="0|1"];
"111":n -> "111":s [ penwidth=2, label ="1|1"];
"111" -> "112" [ penwidth=2, label ="2|1"];
"112" -> "120" [ penwidth=2, label ="0|1"];
"112" -> "121" [ penwidth=2, label ="1|1"];
"112" -> "122" [ penwidth=2, label ="2|1"];
"11(-1)" -> "00(-1)" [ penwidth=2, label ="0|0"];
"11(-1)" -> "000" [ penwidth=2, label ="1|0"];
"11(-1)" -> "001" [ penwidth=2, label ="2|0"];
"11(-2)" -> "0(-1)(-1)" [ penwidth=2, label ="0|0"];
"11(-2)" -> "0(-1)0" [ penwidth=2, label ="1|0"];
"11(-2)" -> "0(-1)1" [ penwidth=2, label ="2|0"];
"120" -> "11(-1)" [ penwidth=2, label ="0|0"];
"120" -> "110" [ penwidth=2, label ="1|0"];
"120" -> "111" [ penwidth=2, label ="2|0"];
"121" -> "10(-1)" [ penwidth=2, label ="00|00"];
"121" -> "100" [ penwidth=2, label ="01|00"];
"121" -> "011" [ penwidth=2, label ="02|11"];
"121" -> "120" [ penwidth=2, label ="1|0"];
"121":n -> "121":s [ penwidth=2, label ="2|0"];
"122" -> "220" [ penwidth=2, label ="0|1"];
"122" -> "21(-1)" [ penwidth=2, label ="10|00"];
"122" -> "120" [ penwidth=2, label ="11|11"];
"122" -> "121" [ penwidth=2, label ="12|11"];
"122" -> "131" [ penwidth=2, label ="2|0"];
"123" -> "230" [ penwidth=2, label ="0|1"];
"123" -> "231" [ penwidth=2, label ="1|1"];
"123" -> "232" [ penwidth=2, label ="2|1"];
"12(-1)" -> "10(-1)" [ penwidth=2, label ="0|0"];
"12(-1)" -> "100" [ penwidth=2, label ="1|0"];
"12(-1)" -> "101" [ penwidth=2, label ="2|0"];
"131" -> "22(-1)" [ penwidth=2, label ="0|0"];
"131" -> "220" [ penwidth=2, label ="1|0"];
"131" -> "221" [ penwidth=2, label ="2|0"];
"132" -> "23(-1)" [ penwidth=2, label ="0|0"];
"132" -> "230" [ penwidth=2, label ="1|0"];
"132" -> "231" [ penwidth=2, label ="2|0"];
"21(-1)" -> "00(-1)" [ penwidth=2, label ="0|1"];
"21(-1)" -> "000" [ penwidth=2, label ="1|1"];
"21(-1)" -> "001" [ penwidth=2, label ="2|1"];
"220" -> "11(-1)" [ penwidth=2, label ="0|1"];
"220" -> "110" [ penwidth=2, label ="1|1"];
"220" -> "111" [ penwidth=2, label ="2|1"];
"221" -> "12(-1)" [ penwidth=2, label ="0|1"];
"221" -> "120" [ penwidth=2, label ="1|1"];
"221" -> "121" [ penwidth=2, label ="2|1"];
"22(-1)" -> "10(-1)" [ penwidth=2, label ="0|1"];
"22(-1)" -> "100" [ penwidth=2, label ="1|1"];
"22(-1)" -> "100" [ penwidth=2, label ="20|00"];
"22(-1)" -> "011" [ penwidth=2, label ="21|11"];
"22(-1)" -> "012" [ penwidth=2, label ="22|11"];
"230" -> "21(-1)" [ penwidth=2, label ="0|1"];
"230" -> "10(-1)" [ penwidth=2, label ="10|00"];
"230" -> "100" [ penwidth=2, label ="11|00"];
"230" -> "011" [ penwidth=2, label ="12|11"];
"230" -> "120" [ penwidth=2, label ="2|0"];
"231" -> "22(-1)" [ penwidth=2, label ="0|1"];
"231" -> "220" [ penwidth=2, label ="1|1"];
"231" -> "21(-1)" [ penwidth=2, label ="20|00"];
"231" -> "120" [ penwidth=2, label ="21|11"];
"231" -> "121" [ penwidth=2, label ="22|11"];
"232" -> "23(-1)" [ penwidth=2, label ="0|1"];
"232" -> "230" [ penwidth=2, label ="1|1"];
"232" -> "231" [ penwidth=2, label ="2|1"];
"23(-1)" -> "11(-2)" [ penwidth=2, label ="0|0"];
"23(-1)" -> "11(-1)" [ penwidth=2, label ="1|0"];
"23(-1)" -> "110" [ penwidth=2, label ="2|0"];
"(-1)03" -> "121" [ penwidth=2, label ="0|0"];
"(-1)03" -> "122" [ penwidth=2, label ="1|0"];
"(-1)03" -> "123" [ penwidth=2, label ="2|0"];
"(-1)(-1)1" -> "001" [ penwidth=2, label ="0|0"];
"(-1)(-1)1" -> "002" [ penwidth=2, label ="1|0"];
"(-1)(-1)1" -> "003" [ penwidth=2, label ="2|0"];
"(-1)(-1)2" -> "011" [ penwidth=2, label ="0|0"];
"(-1)(-1)2" -> "012" [ penwidth=2, label ="1|0"];
"(-1)(-1)2" -> "013" [ penwidth=2, label ="2|0"];

\end{multicols}
\end{centering}
}
\begin{figure}[h!]\centering
\includegraphics[width=\linewidth]{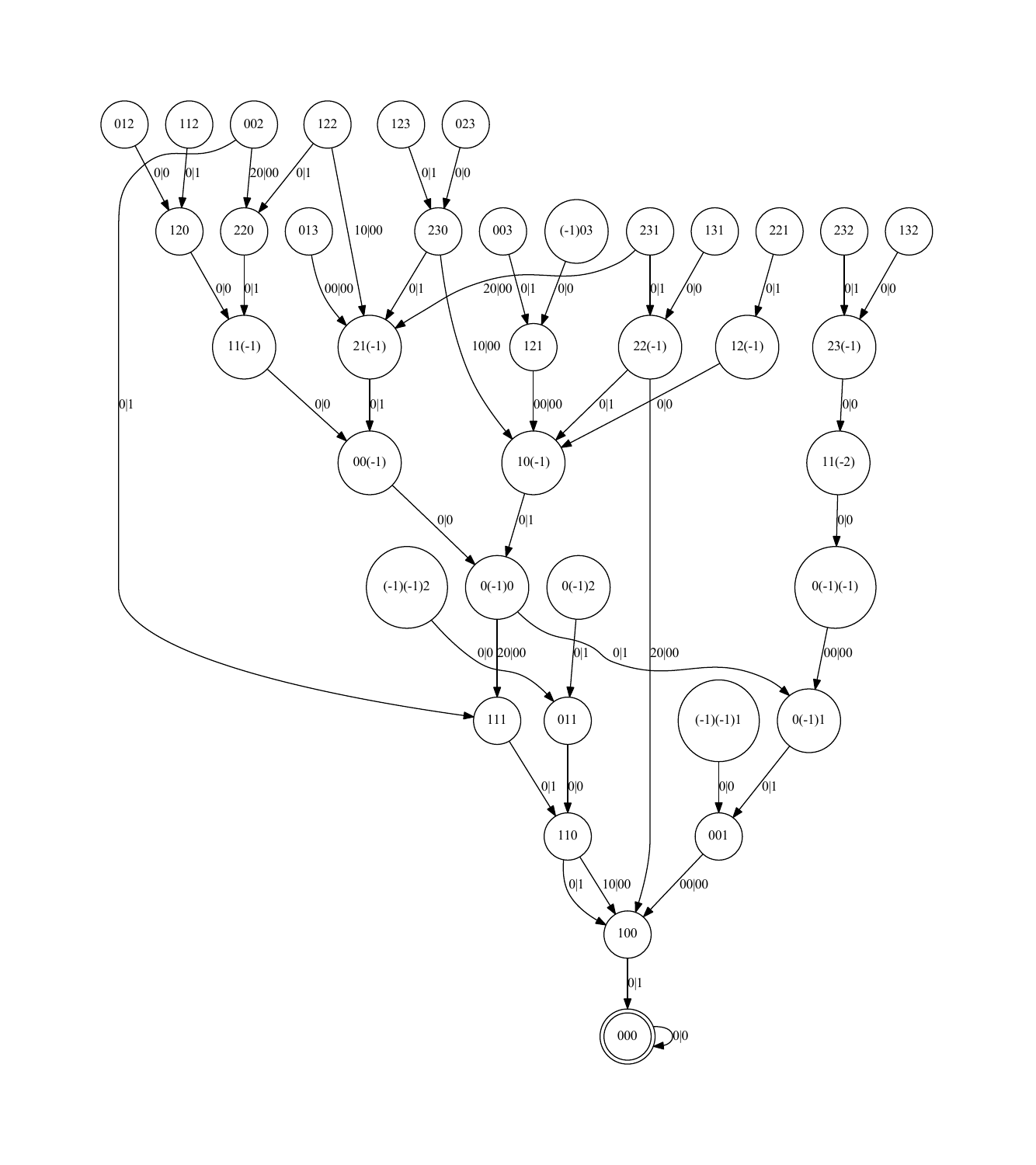}
\caption{Subgraph of $\T_+$ induced by reading zeros at the input}
\label{tree}
\end{figure}
\subsection{Transducer $\T_-$}
Now we define the letter-to-letter transducer $\T_-=(Q,\A,\B,E,I)$, where
\begin{itemize}
  \item $Q\subset \{\overline 2, \overline 1, 0, 1\}^3$ is the set of states;
  \item $I=000$ is the initial state;
  \item $\A=\{0,\overline 1\}$ is the input alphabet;
  \item $\B=\{0,1\}$ is the output alphabet;
  \item The set of transitions is defined below
\end{itemize}
\begin{centering}
\begin{multicols}{3}
$000|0 \rightarrow 0|000$\\
$000|\i \rightarrow 0|00\i$\\
$00\i|0 \rightarrow 0|0\i0$\\
$00\i|\i \rightarrow 0|0\i\i$\\
$0\i\i|0 \rightarrow 0|\i\i0$\\
$0\i\i|\i \rightarrow 1|0\overline 20$\\
$0\overline 20|0 \rightarrow 1|\i\i1$\\
$0\overline 20|\i \rightarrow 1|\i\i0$\\
$\i\i1|0 \rightarrow 0|001$\\
$\i\i1|\i \rightarrow 0|000$\\
$\i\i0|0 \rightarrow 0|0\i1$\\
$\i\i0|\i \rightarrow 0|0\i0$\\
$0\i0|0 \rightarrow 1|0\i1$\\
$0\i0|\i \rightarrow 1|0\i0$\\
$0\i1|0 \rightarrow 1|001$\\
$0\i1|\i \rightarrow 1|000$\\
$001|0 \rightarrow 0|010$\\
$001|\i \rightarrow 1|100$\\
$100|0 \rightarrow 1|000$\\
$100|\i \rightarrow 1|00\i$\\
$010|0 \rightarrow 0|100$\\
$010|\i \rightarrow 0|10\i$\\
$10\i|0 \rightarrow 1|0\i0$\\
$10\i|\i \rightarrow 1|0\i\i$
\end{multicols}
\end{centering}
%\begin{figure}[h!]\centering
%\includegraphics[width=\linewidth]{odcitani}
%\caption{Transducer $\T_-$}
%\label{odcitani}
%\end{figure}
\subsection{Proof of Lemma \ref{lemmatecko}}
\begin{proof}
We will provide demonstration for items 1-5 of the lemma.
\begin{enumerate}
\item Let us assume that the transducer is in the state $q$ and the string $0^\omega$ is on the input. This corresponds to the situation when we start reading the fractional digits. Only the paths starting in
the states
$012,112,122,123,023,003,\i03,$ $232,132$ result in six fractional digits (see Figure~\ref{tree}). Note that in this case, the last nonzero digit is $9$th on the output since first three output digits belong to the integer part of the output.
\item The states listed in the proof of item 1 are accessible only by reading the digit $2$.
\item From Figure~\ref{tree}, one can see that if $0010^\omega$ should not be suffix of $\T_+(x),$ then the path to the state $000$ must go through $110\xrightarrow{0|1}100$. If the path starts in $002, \i\i2,0\i2$, only one fractional digit is produced. The same holds if the path goes through $0\i0$, since we must read $20$ in the transition $0\i0\xrightarrow{20|00}111.$ This case corresponds to the rewriting $0\i02\bullet \rightarrow 001\bullet 1.$
%Also, in case we do not go throught the path
%$001\xrightarrow{00|00}100\xrightarrow{0|1}000$, we obtain at most one fractional digit.
We also obtain at most one fractional digit.
\item Consider transducer $\T_-$ and similarly as in the proof of item 1, inspect the paths leading to $000.$ Precisely, it takes at most six zeros on the input to reach the state $000$.
\item By the same argumentation as in the proof of item 3, one can see that $\T_-(\overline y)$ has the suffix $0010^\omega$ if $\fr_{\T_-}(\overline y)\geq2$.
\end{enumerate}
\end{proof}
%%%%%%%%%%%%%%%%%%%%%%%%%%%%%%%%%%%%%%%%
\section{Property (-F) of cubic confluent Parry numbers}\label{sekcepet}
Let us study Property (-F) of the zeros of $x^3-mx^2-mx-n,\ m,n\in\N,\ m\geq n\geq1.$ Such numbers  belong to the class~\eqref{baze}, sometimes called confluent Parry numbers. Let us recall that in this case $\lfloor\beta\rfloor=m$ and $\sum_{i\in\Z} x_i(-\beta)^i\in\Finmb$ for any sequence
$(x_i)_{i\in\Z}\in \{0,1,\dots,m\}$ with finitely many nonzero elements. It was shown in~\cite{DMV} that the zeros of $x^3-mx^2-mx-n,\ m>n\geq1$ do not possess Property (-F), hence we focus only on the case $m=n.$ Property (-F) for $\gamma$ already follows from Section~\ref{transducers}. For if $0^\omega$ is a suffix of the input sequence, the transducer reaches the state $000$ where it has a cycle $000|0\rightarrow0|000.$

The closeness of $\Finmb$ under multiplication follows trivially from the closeness under addition. For any $x,y\in\Finmb$, we can write $x+y=x+y_{k_1}(-\beta)^{k_1}+\dots +y_{k_r}(-\beta)^{k_r},$ where $r$ is the number of nonzero digits in a representation of $y$. Since it holds $(-\beta)^d\Finmb=\Finmb$ for all $d\in\Z,$ it suffices to show that $-1,x+1\in\Finmb$ 
for any $x\in\Finmb.$ Note that $-1\bullet=1m\bullet00m\in\Finmb$, and $x+1\in\Finmb$ trivially if
$x=x_k\cdots x_0\bullet x_{-1}\cdots,\ x_0\neq m.$ Hence we only need to show that 
$x+1=x_k\cdots x_1 (m+1)\bullet x_{-1}\cdots$ has a finite representation.

In order to show that $x+1\in\Finmb$, we again construct an automaton rewriting a representation of $x+1$ to the alphabet $A=\{0,1,\dots,m\}.$ The following reasoning is helpful for an easy construction. We make use of the principle of the ``sliding window'' moving from left to right, see Example~\ref{priklad}. In each step, a representation of zero is added. Let us recall that it follows from the minimal polynomial of $\beta$, that the digitwise addition of strings 
$1m\overline mm=\i\overline m m\overline m$ does not change the value
of a representation.  Moreover, we do not consider output digits as long as they belong to $A$. The initial setting of the window will be such that it contains the digit $(m+1)$. Obviously, the automaton can stop anytime when the digits in the window belong to $A$.

The states of the automaton are divided into several sets $Q_i$ listed below, not necessarily disjoint. In each set of states, we consider all the choices $x,y,z\in\{0,1,\dots, m\}$ that satisfy additional conditions. These conditions are of two types. \underline{Underlined} conditions may not be satisfied, however, if they are not fulfilled, we just found a finite representation over $A$ and the automaton stops. The other conditions are either assumptions of the initial configuration, or inherited from the previous states.
We also assume that the input digit is any member of $A.$ For a given representation of
$x+1=x_k\cdots x_1 (m+1)\bullet x_{-1}\cdots,$ we have the following initial states:
\begin{enumerate}
  \item If $x_1\geq 1$, the initial state lies in $Q_\mathrm I$.
  \item If $x_2<m, x_1=0$, the initial state lies in $Q_1$.
  \item If $x_4x_3x_2x_1x_0=x_4x_3m0(m+1),$ then it necessarily holds $x_4x_3\neq m0$, otherwise $x$ is a non-admissible representation of $x$. In case $x_3\geq 1,$ we have $x_3m0(m+1)=(x_3-1)0m1$ and we found a finite representation of $x+1.$ In other case we have $x_3=0, x_4<m$ and we rewrite
$x_40m0(m+1)=(x_4+1)m0m(m+1)$ which corresponds to a state in $Q_\mathrm I.$
\end{enumerate}

\begin{ex}
Let $\beta>1$ be the zero of $x^3-2x^2-2x-2.$ Consider the representation $x+1=10203\bullet 10022.$ First, according to item 3 above, we rewrite $10203\bullet 10022=22023\bullet 10022$, hence we start in the state 
$(2,3,1)\in Q_\mathrm I.$
\setlength{\arraycolsep}{3.5pt}
\begin{displaymath}
\begin{array}{ccccc@{\ \ \bullet\ \ }ccccccc|l}
2&2&0&[2&3&1&\b 0]&0&1&2&0&0&Q_\mathrm I\xrightarrow\m Q_\mathrm{II}\subseteq Q_4\\
2&2&0&1&[1&3&\overline 2&\b0]&1&2&0&0&Q_4\xrightarrow{\m} Q_5\\
2&2&0&1&0&[1& 2&{\overline 2}&\b1]&2&0&0&Q_5\xrightarrow{\m} Q_7\\
2&2&0&1&0&0& [0&0&\i&\b2]&0&0&Q_7\xrightarrow{0000} Q_8\\
2&2&0&1&0&0& 0&[0&\i&2&\b0]&0&Q_8\xrightarrow{\p} Q_9\\
2&2&0&1&0&0& 0&1&[1&0&2&\b0]&Q_9: (y=m\wedge z=0)\Rightarrow\mathrm{\bf{STOP}}\\
\end{array}
\end{displaymath}
Note that $y$ and $z$ are not the values in the window, they are the values in the original representation.
\end{ex}
Now we define the transducer by the set of states $Q_i$ together with the transitions and corresponding rewriting rules.
Note that several times it holds $Q_i\subseteq Q_j,$ for instance for $(i,j)=(2,10),(7,5),(1,9).$
We treat them separately for the sake of simplicity of \hbox{verification.}\vspace{0cm}
%$Q_{\mathrm{II}}\subseteq Q_4$, $Q_2,Q_{10.1}\subseteq Q_{10}$
%$Q_1\subseteq Q_9$ $Q_3\subseteq Q_{11}$ $Q_7\subseteq Q_5$, 
\begin{displaymath}
\begin{array}{l@{:\ }ll}
Q_\mathrm I&\{(x,m+1,z):x\neq 0\}&\xrightarrow{\m} Q_{II} \\
Q_{\mathrm{II}}&\{(1,y+m,z-m):\underline{(y\neq 0\vee z\neq m)}\}&\subseteq Q_{4}\\
Q_1&\{(x,0,m+1):x\neq m\}&\xrightarrow{\p} Q_2\\
Q_2&\{(m,1,z+m):\underline{z\neq 0}\}&\xrightarrow{0000} Q_{3}\\
Q_3&\{(1,y+m,z):y\neq0\}&\xrightarrow{\m} Q_{4}\\
Q_4&\{(x,y+m,z-m):x\neq0, \underline{(y\neq 0\vee z\neq m)}\}&\xrightarrow{\m} Q_{5}\\
Q_5&\{(x,y,z-m):{(x\neq 0\vee y\neq m)},\underline{z\neq m}\}& 
\hspace{-10pt}
\begin{cases}
\xrightarrow{0000}Q_{6}&\text{if }y\neq m\\ 
\xrightarrow{\m}Q_{7}&\text{if }y=m
\end{cases}\\
Q_6&\{(x,y-m,z):x\neq m,{y\neq m}\}&\xrightarrow{\p} Q_{9}\\
Q_7&\{(0,y,z-m):y\neq m, \underline{z\neq m}\}&\xrightarrow{0000} Q_{8}\\
Q_8&\{(x,y-m,z):x\neq m, {y\neq m}\}&\xrightarrow{\p} Q_{9}\\
Q_9&\{(x,y-m,z+m):x\neq m, \underline{(y\neq m \vee z\neq 0)}\}&\xrightarrow{\p} Q_{10}\\
Q_{10}&\{(x,y,z+m): (x\neq m\vee y\neq 0), \underline{z\neq 0}\}&%\xrightarrow{} Q_{11}\\
\hspace{-10pt}
\begin{cases}
\xrightarrow{\p}Q_{11}&\text{if }y=0\\ 
\xrightarrow{\m}Q_{12}&\text{if }y\neq 0
\end{cases}\\
Q_{11}&\{(m,y,z+m):y\neq 0,\underline{z\neq 0}\}&\xrightarrow{0000} Q_{12}\\
Q_{12}&\{(x,y+m,z):x\neq 0,{y\neq 0}\}&\xrightarrow{\m} Q_{13}\\
Q_{13}&\{(x,y+m,z-m):x\neq 0,\underline{(y\neq 0\vee z\neq m)}\}&= Q_4
\end{array}
\end{displaymath}
One can see that the output stays in $A$. It remains to show that the automaton stops for any entry with suffix $0^\omega$. Indeed, there is a unique loop
$Q_4\rightarrow\cdots\rightarrow Q_{10}\rightarrow \cdots\rightarrow Q_{13}=Q_4$, and ultimately, the condition $z\neq0$ in $Q_{10}$ is not fulfilled. We have thus proved the following theorem.
\begin{thm}
  Let $\beta>1$ be the zero of $x^3-mx^2-mx-n,\ m,n\in\N,\ m\geq n\geq 1.$ Then $\Finmb$ is a ring if and only if $n=m.$
\end{thm}
\section{Conclusions}
We extended a class of numbers that are known to possess Property (-F). It seems that this property is satisfied also by the zeros of~\eqref{baze} of arbitrary odd degree, in case that $m=n.$ An approach described in Section~\ref{sekcepet} might lead to obtaining such result. It would be nice to find a more general class satisfying Property (-F), somehow comparable for instance with the bases given in \cite{FRSO} or \cite{HOL}.

\section*{Acknowledgements}
The author would like to thank to Zuzana Mas\'akov\'a for many useful comments.
This work was supported by the Grant Agency of the Czech Technical University in Prague, grant No. SGS14/205/OHK4/3T/14 and Czech Science Foundation grant 13-03538S.
%%%%%%%%%%%%%%%%%%%%%%%%%%%%%%%%%%%%%%%%

\end{document}